\documentclass{article}
\usepackage{amsmath,amssymb,amsthm}
\usepackage{algorithm,algorithmic}
\usepackage[round,sort&compress]{natbib}
\usepackage[dvipdfm]{hyperref}

 \newcommand{\RR}{\mathbb{R}}

\newcommand{\EE}{\mathbb{E}}

 \newcommand{\vu}{\boldsymbol{u}}

 \newcommand{\vdelta}{\boldsymbol{\delta}}
 
 \newcommand{\vepsilon}{\boldsymbol{\epsilon}}

 \newcommand{\tX}{\mathcal{X}}
 
 \newcommand{\tW}{\mathcal{W}}

\newcommand{\norm}[1]{\bigl|\!\bigl|\!\bigl|#1\bigr|\!\bigr|\!\bigr|}

 \def\dot#1#2{\left\langle #1,#2\right\rangle}

\newtheorem{lemma}{Lemma}
\newtheorem{theorem}{Theorem}

\newtheorem{corollary}{Corollary}

\title{Spectral norm of random tensors}
\author{Ryota Tomioka and Taiji Suzuki}
\begin{document}
\maketitle

\begin{abstract}
 We show that the spectral norm of a random $n_1\times n_2\times \cdots
\times n_K$ tensor (or higher-order array) scales as $O\left(\sqrt{(\sum_{k=1}^{K}n_k
 )\log(K)}\right)$ under some sub-Gaussian assumption on the entries. The proof
 is based on a covering number argument. Since the spectral norm is dual
 to the tensor nuclear norm (the tightest convex relaxation of the set
 of rank one tensors), the bound implies that the convex relaxation
 yields sample complexity that is linear in (the sum of) the number of
 dimensions, which is much smaller than other recently proposed convex relaxations of
 tensor rank that use unfolding.
\end{abstract}
\section{Notation and main result}
Let $\tX\in\RR^{n_1\times\cdots\times n_K}$ be a K-way tensor. The
spectral norm of $\tX$ is defined as follows:
\begin{align}
\label{eq:spectralnorm}
 \norm{\tX} =
 \sup_{\vu_1,\vu_2,\ldots,\vu_k}\tX(\vu_1,\vu_2,\ldots,\vu_K)
\quad{\rm s.t.}\quad \vu_k\in S_{n_k-1}\quad (k=1,\ldots,K),
\end{align}
where
$\tX(\vu_1,\ldots,\vu_K)=\sum_{i_1,i_2,\cdots,i_K}X_{i_1i_2\cdots i_K}u_{1i_1}u_{2i_2}\cdots
u_{Ki_K}$ and $S_{n_k-1}$ is the unit sphere in $\RR^{n_k}$.

\begin{lemma}
\label{lem:entrytoproduct}
 Assume that each element $X_{i_1i_2\cdots i_K}$ is independent,
 zero-mean, and satisfies $\EE [e^{tX_{i_1\cdots i_K}}]\leq e^{\sigma^2
 t^2/2}$. Then we have
\begin{align*}
P\left(|\tX(\vu_1,\ldots,\vu_K)|\geq t\right)\leq 2\exp\left(-\frac{t^2}{2\sigma^2}\right),
\end{align*}
if $\vu_k\in S_{n_k-1}$ for $k=1,\ldots, K$.
\end{lemma}
\begin{proof}
By the assumption $E\left[e^{s X_{i_1i_2\cdots i_K}u_{1i_1}u_{2i_2}\cdots
u_{Ki_K}}\right]\leq \exp(u_{1i_1}^2u_{2i_2}^2\cdots
 u_{Ki_K}^2\sigma^2s^2/2)$. Then follow the line of the proof of
 Hoeffding's inequality to obtain 
\begin{align*}
 P\left(\tX(\vu_1,\ldots,\vu_K)\geq
 t\right)&=P\left(e^{s\tX(\vu_1,\ldots,\vu_K)}\geq e^{st}\right)\\
&\leq e^{-st} E\left[e^{s\tX(\vu_1,\ldots,\vu_K)}\right]\\
&\leq
 \exp\biggl\{-st+\frac{\sigma^2s^2}{2}\underbrace{\sum_{i_1=1}^{n_1}u_{1i_1}^2\sum_{i_2=1}^{n_2}u_{2i_2}^2\cdots\sum_{i_K=1}^{n_K}u_{Ki_K}^2}_{=1}\biggr\}\\
&=\exp\left(-st+\frac{\sigma^2s^2}{2}\right).
\end{align*}
Minimizing over $s$, the right-hand side becomes
 $e^{-t^2/(2\sigma^2)}$. Similarly we obtain $P(\tX(\vu_1,\ldots,\vu_K)\leq
 -t)\leq e^{-t^2/(2\sigma^2)}$, and the statement is obtained by taking
 the union of the two cases.
\end{proof}

\begin{theorem}
\label{thm:specnormbound}
 Assume that for each fixed $\vu_k\in S_k$ ($k=1,\ldots,K$), we have
\begin{align*}
P\left(|\tX(\vu_1,\ldots,\vu_K)|\geq t\right)\leq 2\exp\left(-\frac{t^2}{2\sigma^2}\right).
\end{align*}
 Then the spectral norm $\norm{\tX}$ can be bounded as follows:
\begin{align*}
 \norm{\tX}\leq\sqrt{8\sigma^2\left(\left(\sum\nolimits_{k=1}^{K}n_k\right)\log(2K/K_0)+\log(2/\delta)\right)},
\end{align*}
with probability at least $1-\delta$ and $K_0=\log(3/2)$.
\end{theorem}
\begin{proof}
We use a covering number argument. Let $C_1,\ldots,C_K$ be $\epsilon$-covers
 of $S^{n_1-1}, \ldots, S^{n_K-1}$. Then since
 $S^{n_1-1}\times\cdots\times S^{n_K-1}$ is compact, there is a
 maximizer $(\vu_1^\ast,\ldots,\vu_K^\ast)$ of \eqref{eq:spectralnorm}
 and using the $\epsilon$-covers, we can write
\begin{align*}
 \norm{\tX}=\tX(\bar{\vu}_1+\vdelta_1,\bar{\vu}_2+\vdelta_2,\ldots,\bar{\vu}_K+\vdelta_K),
\end{align*}
where $\bar{\vu}_k\in C_k$ and $\|\vdelta_k\|\leq \epsilon$ for
 $k=1,\ldots,K$ by the definition.
Now
\begin{align*}
 \norm{\tX}\leq&\tX(\bar{\vu}_1,\ldots,\bar{\vu}_K)+\left(\epsilon K + \epsilon^2\binom{K}{2}+\cdots \epsilon^K\binom{K}{K}\right)\norm{\tX}.
\end{align*}
Take $\epsilon=K_0/K$ then the sum inside the parenthesis
 can be bounded as follows:
\begin{align*}
\epsilon K + \epsilon^2\binom{K}{2}+\cdots \epsilon^K\binom{K}{K}&\leq
 \epsilon K + \frac{(\epsilon K)^2}{2!}+\cdots \frac{(\epsilon
 K)^K}{K!}\leq e^{\epsilon K}-1= \frac{1}{2}.
\end{align*}
Thus we have 
\begin{align*}
 \norm{\tX}\leq 2 \max_{\bar{\vu}_1\in C_1,\ldots,\bar{\vu}_K\in C_K}\tX(\bar{\vu}_1,\ldots,\bar{\vu}_K).
\end{align*}
Since the $\epsilon$-covering number $|C_k|$ can be bounded by
 $\epsilon/2$-packing number, which can be bounded by
 $(2/\epsilon)^{n_k}$, using the union bound we obtain
\begin{align*}
 P(\norm{\tX}\geq t)&\leq \sum_{\bar{\vu}_1\in C_1,\ldots,\bar{\vu}_K\in
 C_K}
P\left(\tX(\bar{\vu}_1,\ldots,\bar{\vu}_K)\geq \frac{t}{2}\right)
\\
&\leq \left(\frac{2K}{K_0}\right)^{\sum_{k=1}^{K}n_k}\cdot2\exp\left(-\frac{t^2}{8\sigma^2}\right).
\end{align*}
Finally, we take
 $t=\sqrt{8\sigma^2\left((\sum_{k}n_k)\log(2K/K_0)+\log(2/\delta)\right)}$
 to obtain our claim.
\end{proof}

We note that a similar bound was proved in \cite{NguDriTra10}. We
believe that our proof is more concise and simple.

\subsection{Implication for tensor recovery with Gaussian measurements}
\begin{corollary}
Assume that each entry $X_{i_1\cdots i_K}$ is conditionally independent
 given $\vepsilon=(\epsilon_i)_{i=1}^{M}$ and distributed as
\begin{align*}
 X_{i_1\cdots i_K} = \sum_{j=1}^{M}\epsilon_j W_{ji_1i_2\cdots i_K},
\end{align*}
where each $W_{ji_1i_2\cdots i_K}$ is independent, zero-mean, and
 satisfies $\EE[e^{tW_{ji_1i_2\cdots i_K}}]\leq \exp(t^2/2)$; in
 addition, each $\epsilon_i$ is also independent, zero-mean and satisfes
 $\EE[e^{t\epsilon_i}]\leq \exp(\sigma^2t^2/2)$. If $M\geq
 2\log(2/\delta)$, then with probability at least $1-\delta$, we have
\begin{align*}
 \norm{\tX}\leq \sqrt{32M\sigma^2\left(\sum_{k=1}^{K}n_k\log(2K/K_0)+\log(4/\delta)\right)}
\end{align*}
\end{corollary}
\begin{proof}
 Conditioned on $\vepsilon$, the moment generating function
 $\EE[\exp(t\tX(\vu_1,\ldots,\vu_K))]$ can be bounded as follows:
\begin{align*}
\EE\left[e^{t\tX(\vu_1,\ldots,\vu_K)}\right]&=\prod_{i_1}\cdots\prod_{i_K}\prod_{j}\EE\left[e^{t\epsilon_ju_{1i_1}\cdots
 u_{Ki_K}W_{ji_1\ldots i_K}}\right]\\
&\leq \exp\left(\frac{\|\vepsilon\|^2 t^2}{2}\right),
\end{align*}
where we used the fact that $\sum_{i_1}\cdots\sum_{i_K}u_{1i_1}^2\cdots
 u_{Ki_K}^2=1$. Therefore, we have
\begin{align*}
 P(\left.\left|\tX(\vu_1,\ldots,\vu_K)\right|\geq t\right|\vepsilon)\leq 2
\exp\left(-\frac{t^2}{2\|\vepsilon\|^2}\right),
\end{align*}
using Hoeffding's inequality.

Now we can apply Theorem \ref{thm:specnormbound} as follows:
\begin{align*}
 P(\norm{\tX}\geq t)=&P\left(\left.\norm{\tX}\geq t\right|\|\vepsilon\|\leq
 2\sqrt{M\sigma^2}\right)\underbrace{P(\|\vepsilon\|\leq
 2\sqrt{M\sigma^2})}_{\leq 1}\\
&+\underbrace{P\left(\left.\norm{\tX}\geq t\right|\|\vepsilon\|>
 2\sqrt{M\sigma^2}\right)}_{\leq 1} P(\|\vepsilon\|>2\sqrt{M\sigma^2})\\
\leq &
 \left(\frac{2K}{K_0}\right)^{\sum_{k=1}^{K}n_k}\cdot2\exp\left(-\frac{t^2}{32M\sigma^2}\right)
 + \exp\left(-\frac{M}{2}\right)\\
\leq & \frac{\delta}{2}+\frac{\delta}{2}=\delta.
\end{align*}
\end{proof}

\subsection{Implication for sampling without replacement}
\begin{corollary}
 Suppose $\tX$ contains $M$ nonzero entries sampled uniformly
 without replacement; each entry is a random variable $\epsilon_j$
 $(j=1,\ldots,M)$ that satisfies $\EE[e^{t\epsilon_j}]\leq
 \exp(\sigma^2t^2/2)$. Then we have
\begin{align*}
 \norm{\tX}\leq\sqrt{8\sigma^2\left(\left(\sum\nolimits_{k=1}^{K}n_k\right)\log(2K/K_0)+\log(2/\delta)\right)},
\end{align*}
with probability at least $1-\delta$ and $K_0=\log(3/2)$. 
\end{corollary}
\begin{proof}
This is analogous to the proof of Lemma 4 in \cite{RohTsy11}. Let
 $\tW_1,\ldots,\tW_M$ be tensors that each are an indicator of the observed
 positions. Then $\tX=\sum_{j=1}^{M}\epsilon_j\tW_j$. Since each entry
 is observed maximally once, we have
\begin{align*}
 \sum_{j=1}^{M}\tW_j^2(\vu_1,\ldots,\vu_K)=\sum_{j=1}^{M}\dot{\tW_j}{\vu_1\circ
 \vu_2\circ\cdots\circ \vu_K}^2\leq \|\vu_1\circ\cdots\circ\vu_K\|_F^2=1.
\end{align*}
Thus using Hoeffding's inequality
\begin{align*}
 P\left(|\tX(\vu_1,\ldots,\vu_K)|\geq t|(\tW_j)\right)\leq 2\exp\left(-\frac{t^2}{2\sigma^2}\right).
\end{align*}
Taking expectation over the choice of $\tW_j$ $(j=1,\ldots,M)$, we obtain
\begin{align*}
 P\left(|\tX(\vu_1,\ldots,\vu_K)|\geq t\right)\leq 2\exp\left(-\frac{t^2}{2\sigma^2}\right).
\end{align*}
The claim now follows from Theorem~\ref{thm:specnormbound}.
\end{proof}

\bibliographystyle{abbrvnat}
\bibliography{icml2014}

\begin{thebibliography}{2}
\providecommand{\natexlab}[1]{#1}
\providecommand{\url}[1]{\texttt{#1}}
\expandafter\ifx\csname urlstyle\endcsname\relax
  \providecommand{\doi}[1]{doi: #1}\else
  \providecommand{\doi}{doi: \begingroup \urlstyle{rm}\Url}\fi

\bibitem[Nguyen et~al.(2010)Nguyen, Drineas, and Tran]{NguDriTra10}
N.~H. Nguyen, P.~Drineas, and T.~D. Tran.
\newblock Tensor sparsification via a bound on the spectral norm of random
  tensors.
\newblock Technical report, arXiv:1005.4732, 2010.

\bibitem[Rohde and Tsybakov(2011)]{RohTsy11}
A.~Rohde and A.~B. Tsybakov.
\newblock Estimation of high-dimensional low-rank matrices.
\newblock \emph{The Annals of Statistics}, 39\penalty0 (2):\penalty0 887--930,
  2011.

\end{thebibliography}
\end{document}